\documentclass[letterpaper, 10 pt, conference]{ieeeconf}
\IEEEoverridecommandlockouts                              
\usepackage{cite}
 \usepackage[inkscapearea=page]{svg}
\usepackage{graphicx}      
\usepackage{amsmath}
\usepackage{amssymb}
\usepackage{subcaption}

\usepackage{amsthm}
\usepackage{mathrsfs}
\usepackage{empheq}
\usepackage{tikz}
\usepackage[utf8]{inputenc}
\usepackage{pgfplots} 
\usepackage{pgfgantt}
\usepackage{pdflscape}
\pgfplotsset{compat=newest} 
\pgfplotsset{plot coordinates/math parser=false}
\usepackage{tikzscale}
\usepackage{standalone}
\usetikzlibrary{shapes,arrows}
\usetikzlibrary{arrows,calc,positioning,fit}
\usepackage{amsmath}
\let\labelindent\relax
\usepackage{enumitem} 

\usepackage{stfloats}

\newtheorem{theorem}{Theorem}
\newtheorem{definition}{Definition}

\newtheorem{remark}{Remark}
\newtheorem{lemma}{Lemma}
\newtheorem{corollary}{Corollary}

\def\BibTeX{{\rm B\kern-.05em{\sc i\kern-.025em b}\kern-.08em
    T\kern-.1667em\lower.7ex\hbox{E}\kern-.125emX}}

\usepackage{eso-pic}
\AddToShipoutPictureBG*{%
\AtPageUpperLeft{%
\setlength\unitlength{1in}%
\hspace*{\dimexpr0.5\paperwidth\relax}
\makebox(0,-1.2)[c]{
\begin{tabular}{c c}
T. de Groot \emph{et al.},
Scaled Relative Graphs and Dynamic Integral Quadratic Constraints: \\Connections and Computations for Nonlinear Systems, \\
uploaded to arXiv  02 April 2026 \\
\end{tabular}}}}

\begin{document}

\title{\LARGE \bf
Scaled Relative Graphs and Dynamic Integral Quadratic Constraints: Connections and Computations for Nonlinear Systems
}
\author{Timo de Groot, Tom Oomen, W.P.M.H. Heemels, Sebastiaan van den Eijnden 
\thanks{The authors are with the Eindhoven University of Technology, Department of Mechanical Engineering, P.O. Box 513, 5600MB Eindhoven, The Netherlands (e-mail: s.j.a.m.v.d.eijnden@tue.nl).}
\thanks{Research partly supported by the European Union (ERC Advanced Grant, Proacthis, no. 10105538).}}

\maketitle

\begin{abstract}
Scaled relative graphs (SRGs) enable graphical analysis and design of nonlinear systems. In this paper, we present a systematic approach for computing both soft and hard SRGs of nonlinear systems using dynamic integral quadratic constraints (IQCs). These constraints are exploited via application of the S-procedure to compute tractable SRG overbounds. In particular, we show that the multipliers associated with the IQCs define regions in the complex plane. Soft SRG computations are formulated through frequency-domain conditions, while hard SRGs are obtained via hard factorizations of multipliers and linear matrix inequalities. The overbounds are used to derive an SRG-based feedback stability result for Lur'e-type systems, providing a new graphical interpretation of classical IQC stability results with dynamic multipliers.  
\end{abstract}

\section{Introduction}
Scaled relative graphs (SRGs) enable graphical analysis and robust design of nonlinear control systems, complementing powerful and widely adopted loop-shaping techniques for linear time-invariant (LTI) systems. Since their introduction in the optimization literature \cite{Ryu22} and the systems and control community \cite{Chaffey23}, SRGs have seen rapid developments, both in theoretical foundations and practical application \cite{Barron25, Chen, Krebbex, Groot, Eijnden24, Evans, Barron25_2, Eijnden25, Zhang2}. 

Practical use of SRGs for system analysis and control design requires efficient computation of a system's SRG. For LTI systems, several computational approaches have been reported in the literature, exploiting, for instance, links with the Nyquist diagram \cite[Section 5]{Chaffey23}, \cite{Chen25cdc}, gain computations \cite{Nauta, Krebbekx25d}, and dissipativity-based characterizations \cite{Groot}. These methods address both soft and hard SRGs, the latter which are particularly useful for studying unbounded and unstable systems, including systems with integrators. The dissipativity-based approach has been extended to soft and hard SRG computations for specific classes of nonlinear systems exhibiting ``linear-like'' dynamics, including piecewise linear systems,  reset control systems, and Lur'e systems with memoryless, sector-bounded nonlinearities \cite{Groot, Eijnden24, Groot25cdc}. However, these methods do not directly extend to more general nonlinear systems.

In this paper, we address the open problem of computing soft and hard SRGs for a broader class of nonlinear systems by presenting a method for constructing SRG overbounds. Our approach relies on the use of (dynamic) IQCs \cite{Megretski, Carrasco_seiler} and builds on the ideas in \cite{Eijnden26} where regions in the complex plane are characterized by multipliers. Specifically, we assume that the nonlinear systems under consideration are known to satisfy one or multiple soft and/or hard IQCs. Multipliers associated with these IQCs are exploited via the S-procedure to derive additional input-output relationships for the nonlinear system, which are translated into regions in the complex plane that overbound the system's SRG. The role of the IQCs is to replace an exact description of the system dynamics with simpler characterizations that enable tractable SRG computational methods. 

The contributions of this paper are threefold: 

1) We introduce a novel method using IQCs for overbounding SRGs of generic nonlinear dynamical systems, going beyond the ``linear-like'' classes considered so far in existing work. For computing soft SRGs we present frequency-domain conditions, whereas for hard SRGs we exploit the hard factorization of multipliers into stable LTI filters \cite{Carrasco_seiler} to formulate computations based on linear matrix inequalities (LMIs). Links between the proposed soft and hard SRG computations are provided. 

2) We exploit the overbounds to formulate an SRG-based stability result for Lur'e systems, formed by the feedback interconnection of an LTI system and a nonlinearity. We show that this result provides a new graphical interpretation of the classical IQC theorem in \cite{Megretski}, allowing stability to be assessed  directly from the original Nyquist plot of the LTI part instead of modifying it with dynamic multipliers. This parallels graphical tests based on static multipliers, such as the passivity, small-gain, and conicity theorems \cite{Zames66}, as well as the (generalized) circle-criterion \cite{Krebbex}. 

3) We present specific connections with the Popov criterion and Zames-Falb multipliers \cite{Desoer09, Carrasco}, and reveal some structural limitations. Examples are provided throughout the paper to illustrate our main results.

The remainder of this paper is organized as follows. In Section~\ref{sec:SRG_IQC} we provide background on SRGs and IQCs. Section~\ref{sec:main} presents our main result on computing the hard and soft SRGs of nonlinear systems by exploiting dynamic IQCs. A feedback stability theorem is presented in Section~\ref{sec:AS}. Examples are provided throughout these sections, and conclusions are given in Section~\ref{sec:conclusions}. 

\subsection*{Notation} An $n$-by-$n$ Hermitian matrix $P$ is called positive (negative) semi-definite, denoted by $P\succeq 0$ ($P\preceq 0$), if $x^* P x \geq 0$ ($x^*Px \leq 0$) for all $x \in \mathbb{C}^n$, where $*$ denotes the complex conjugate transpose. 
The distance between two sets ${S}_1,S_2\subseteq \mathbb{C}$ is denoted $\textup{dist}({S}_1,{S}_2)=\inf_{a\in {S}_1,\;b\in{S}_2}|a-b|$. 

Denote the set of square-integrable $\mathbb{R}^n$-valued signals by
$$\mathcal{L}_2^n := \left\{u: [0,\infty)\to\mathbb{R}^n \mid \int_{0}^\infty u(t)^\top u(t) dt < \infty\right\},$$ 
and let the inner product and induced norm be {defined by}
$$\langle u,y\rangle := \int_{0}^\infty u(t)^\top y(t) dt \:\textup{ and }\:\|u\| := \sqrt{\langle u,u\rangle}.$$ 
The extended space $\mathcal{L}_{2e}^n$ is given by
\begin{equation*}
      \mathcal{L}_{2e}^n := \left\{u:[0,\infty) \to \mathbb{R}^n \mid \Gamma_T u\in \mathcal{L}_2^n \ \ \textup{for all } T \geq 0\right\},
\end{equation*}
where the truncation operator $\Gamma_T$ on a signal $u:[0,\infty)\to \mathbb{R}^n$ is defined such that $\Gamma_Tu :[0,\infty)\to \mathbb{R}^n$ with $(\Gamma_Tu)(t)=u(t)$ when $t\in [0,T]$, and $(\Gamma_Tu)(t)=0$ when $t> T$. We write $u_T = \Gamma_Tu$ for any $T\geq 0$ when there is no ambiguity. The Fourier transform of a signal $u\in\mathcal{L}_2^n$ is given by $$\hat{u}(j\omega):=\int_{0}^{\infty}u(t)e^{-j\omega t}\,dt.$$

A \emph{system} is defined as an operator $H:\mathcal{L}_{2e}^n \to \mathcal{L}_{2e}^n$. We work with \emph{square} systems, i.e., systems with an equal number of inputs and outputs. Note that non-square systems can always be patched with zeros to make them square, hence this does not form a limitation. We assume $H 0 = 0$, i.e., the zero input produces zero output. The incremental gain of $H:\mathcal{L}_{2e}^n \to \mathcal{L}_{2e}^n$ is given by

 \begin{equation*}
        \|H\|_\Delta := \sup\left\{\tfrac{\|Hu-Hv\|}{\|u-v\|}\mid u,v\in \mathcal{L}_2^n, \:u\neq v\right\}.
    \end{equation*}
   We call a system bounded (also referred to as incrementally stable) if $\|H\|_\Delta <\infty$. Causality of a system means that $\Gamma_THu = \Gamma_TH\Gamma_Tu$ for all $T\geq 0$ and all $u \in \mathcal{L}_{2e}^n$. For an LTI system $H$ we will write with some abuse of notation $H(s)$ and $H(j\omega)$ to denote its transfer function and frequency-response function. The graph and extended graph of a system $H$ are denoted by $
        G(H)=\left\{(u,y)\in \mathcal{L}_2^n\times \mathcal{L}_2^n \mid y= Hu\right\},
   $ and $G_e(H)=\left\{(u,y)\in \mathcal{L}_{2e}^n\times \mathcal{L}_{2e}^n\mid y= Hu\right\}$, respectively. We will use $\Delta(\cdot)$ as a shorthand notation for the incremental signal $(\cdot)_1-(\cdot)_2$, e.g., $\Delta v = v_1-v_2$. 

\section{Background on SRGs and IQCs}\label{sec:SRG_IQC}
In this section, we review background on SRGs and IQCs, and recall useful connections between these concepts. 
\subsection{Scaled Relative Graphs (SRGs)}
SRGs allow for graphically depicting input-output properties of a system $H$ expressed in terms of gain and phase. Given signals $u\in \mathcal{L}_2^n$ and $y\in \mathcal{L}_2^n$, we define the gain $\rho(u,y)\in [0,\infty]$ from $u$ to $y$ by
\begin{equation}\label{eq:gain}
    \rho(u,y):=\begin{cases}
        \|y\|/\|u\|&\textup{if } u\neq 0,\\
        0 & \textup{if } u=0,y=0,\\
        \infty&\textup{if } u = 0, y\neq 0,
    \end{cases}
\end{equation}
and the phase $\theta(u,y)\in [0,\pi]$ from $u$ to $y$ by
\begin{equation}\label{eq:phase}
    \theta(u,y):=\begin{cases}
        \arccos \frac{\langle u,y\rangle}{\|u\|\|y\|}&\textup{if } u,y\neq 0,\\
        0&\textup{otherwise}.
    \end{cases}
\end{equation}
Likewise, the gain and phase for signals $u \in \mathcal{L}_{2e}^n$, $y\in \mathcal{L}_{2e}^n$ and for $T> 0$ are defined by
\begin{equation}
    \rho_T(u,y):=\rho(u_T,y_T) \textup{ and } \theta_T(u,y):=\theta(u_T,y_T),
\end{equation}
as $u_T,y_T\in \mathcal{L}_2^n$ for $T\geq 0$. Equipped with these notions for gain and phase, we recall definitions for soft and hard SRGs. 

\begin{definition}\label{eq:SG}
    \noindent i) For a causal and bounded system $H:\mathcal{L}_2^n \to \mathcal{L}_2^n$ the \emph{soft SRG} is defined as
\begin{equation}\label{eq:sSRG}
\begin{split}
    &\textup{SRG}(H)=\big\{\rho(\Delta u,\Delta y)e^{\pm j\theta(\Delta u,\Delta y)}\mid\\
    & \qquad\qquad\qquad (u_1,y_1),(u_2,y_2)\in G(H), \;\Delta u \neq 0\big\}.
    \end{split}
\end{equation}
ii) For a causal system $H:\mathcal{L}_{2e}^n \to \mathcal{L}_{2e}^n$ the \emph{hard SRG} is defined as
\begin{equation}\label{eq:hSRG}
\begin{split}
    &\textup{SRG}_e(H)=\big\{\rho_T(\Delta u,\Delta y)e^{\pm j\theta_T(\Delta u,\Delta y)}\mid \\
    &\quad (u_1,y_1),(u_2,y_2)\in G_e(H),\;(\Delta u)_T\neq 0,\;T>0\big\}.
    \end{split}
\end{equation}
\end{definition}

For the specific case that $H$ is an LTI system, we additionally define the frequency-wise (soft) SRG as
\begin{equation}\label{eq:fwsrg}
\begin{split}
    \textup{SRG}(H(j\omega)) = &\big\{\rho(u,H(j\omega)u)e^{\pm j\theta(u,H(j\omega)u)} \mid \\
    &\quad\qquad\qquad \qquad \qquad 0\neq u \in \mathbb{C}^{n}\big\}. 
    \end{split}
\end{equation}
In case $H$ is SISO, \eqref{eq:fwsrg} recovers the Nyquist plot of $H$, i.e., the locus of points $\left\{H(j\omega)\mid \omega \in \mathbb{R}\right\}$, see \cite{Chen25cdc}. 

\subsection{Integral Quadratic Constraints (IQCs)}
IQCs provide a way of representing input-output relationships for complex dynamical systems in a form that is convenient for analysis. We next recall notions of soft and hard incremental IQCs \cite{Su2, Carrasco_seiler}.

\begin{definition}\label{def:IQC}
  Consider a causal system $H:\mathcal{L}_{2e}^n \to \mathcal{L}_{2e}^n$. Let $M=M^\top \in \mathbb{R}^{m\times m}$ and let $N:\mathcal{L}_{2}^{n}\times \mathcal{L}_{2}^n \to \mathcal{L}_{2}^{m}$ be a causal and bounded LTI system. 

  \noindent  i) $H$ is said to satisfy the \emph{soft incremental IQC} with multipliers $(M,N)$ if 
  \begin{equation}\label{eq:sIQC}
      \int_{0}^\infty \left(N \begin{bmatrix}
          \Delta u\\\Delta y
      \end{bmatrix} \right)^\top (t)M\left(N \begin{bmatrix}
          \Delta u\\\Delta y
      \end{bmatrix} \right)(t)dt \geq 0
  \end{equation}
  for all $(u_1,y_1), (u_2,y_2) \in G(H)$.
  
  \noindent  ii) $H$ is said to satisfy the \emph{hard incremental IQC} with multipliers $(M,N)$ if 
  \begin{equation}\label{eq:hIQC}
      \int_{0}^T \left(N \begin{bmatrix}
          \Delta u\\\Delta y
      \end{bmatrix} \right)^\top (t)M\left(N \begin{bmatrix}
          \Delta u\\\Delta y
      \end{bmatrix} \right)(t)dt \geq 0
  \end{equation}
  for all $(u_1,y_1), (u_2,y_2) \in G_e(H)$ and for all $T>0$.
\end{definition}

We highlight a few aspects regarding Definition~\ref{def:IQC}.
\begin{enumerate}
    \item It follows directly from Parseval's theorem that the soft IQC in \eqref{eq:sIQC} expressed in time-domain can equivalently be expressed in the frequency-domain as
\begin{equation}
    \int_{-\infty}^\infty \begin{bmatrix}
        \Delta \hat{u}(j\omega)\\\Delta \hat{y}(j\omega)
    \end{bmatrix} ^*\Theta(j\omega)\begin{bmatrix}
        \Delta \hat{u}(j\omega)\\\Delta \hat{y}(j\omega)
    \end{bmatrix}d\omega \geq 0,
\end{equation}
with $\Theta(j\omega) = N(j\omega)^*MN(j\omega)$ being a bounded, Hermitian-valued matrix. 
\item The IQC representation in \eqref{eq:hIQC} in terms of $M$ and $N$ is also known as a hard factorization \cite{Carrasco_seiler}. In the upcoming approach, we assume that $N$ admits the realization
\begin{equation}\label{eq:Ns}
  N(s)=C_N(sI-A_N)^{-1}B_N+D_N,
\end{equation}
with $(A_N,B_N,C_N,D_N)$ of appropriate dimensions, and $A_N$ Hurwitz (i.e., $N$ is bounded). 
\item When $N=I$, \eqref{eq:sIQC} and \eqref{eq:hIQC} reduce to static IQCs defined by the static multiplier $M$. Classical examples include 
\begin{equation}
   M = \begin{bmatrix}
        \gamma^2 I & 0\\
        0 & -I
    \end{bmatrix}, \:\:\textup{ and }\:\:  M=\begin{bmatrix}
        0&I\\
        I&0
    \end{bmatrix}, 
\end{equation}
representing incremental gain and passivity properties. 
\end{enumerate}

\subsection{Connections between SRGs and static IQCs}
Although distinct in form, there exists a connection between SRGs and \emph{static} IQCs, which we recall here from \cite{Groot}. 

\begin{lemma}\label{lem:SRGIQC}
      Let $N=I$ and $M = \Pi \otimes I$ in \eqref{eq:sIQC} and \eqref{eq:hIQC}, where $\Pi = \Pi^\top \in \mathbb{R}^{2\times 2}$ is given. A causal system $H:\mathcal{L}_{2e}^n \to \mathcal{L}_{2e}^n$ satisfies the soft incremental IQC in \eqref{eq:sIQC} if and only if $\textup{SRG}(H) \subseteq \mathcal{S}(\Pi)$, with
\begin{align} \label{eq:SPI}
	\mathcal{S}(\Pi) := \left\{ z \in\mathbb{C} \left|  \begin{bmatrix}
		1 \\ z
	\end{bmatrix}^* \Pi \begin{bmatrix}
		1 \\ z
	\end{bmatrix} \geq 0 \right.  \right\}.
\end{align} 
Furthermore, $H$ satisfies the hard incremental IQC in \eqref{eq:hIQC} if and only if $\textup{SRG}_e(H) \subseteq \mathcal{S}(\Pi)$.
\end{lemma}

Lemma~\ref{lem:SRGIQC} characterizes SRG-fullness: the class of systems satisfying the static IQC has its SRG contained in $\mathcal{S}(\Pi)$, and, conversely, any system whose SRG is contained in $\mathcal{S}(\Pi)$ satisfies the static IQC. This equivalence allows class membership to be tested graphically, and enables tight overbounding of the SRG for an entire system class. The shape of the overbounding region depends on the static multiplier $\Pi$. When $\det(\Pi)\geq 0$, this region reduces to either the entire complex plane, the empty set, or a circle (or single point). When $\det(\Pi)<0$, this region corresponds to a half-plane or the interior/exterior of a disk centered on the real axis.

Motivated by the characterization in Lemma~\ref{lem:SRGIQC} for static IQCs, in the next section we develop related connections between SRGs and \emph{dynamic} IQCs, and exploit these results to overbound the SRGs of (classes of) nonlinear systems.

\section{SRG overbounding via dynamic IQCs}\label{sec:main}
This section presents our first main result in the form of soft and hard SRG computations based on dynamic IQCs.
\subsection{Overbounding soft SRGs}
We first present a result for overbounding soft SRGs.

\begin{theorem}\label{th:sSRG1}
      Consider a causal and bounded system $H:\mathcal{L}_{2e}^n\to \mathcal{L}_{2e}^n$ and suppose that this system satisfies a soft incremental IQC of the form as in \eqref{eq:sIQC} with multipliers $(M,N)$. If there exist a symmetric matrix $\Pi \in\mathbb{R}^{2\times 2}$, and $\tau> 0$ such that the frequency-domain matrix inequality
\begin{equation}\label{eq:FDI}
\begin{bmatrix}
    I \\ N(j\omega)
\end{bmatrix}^*\begin{bmatrix}
    \Pi \otimes I & 0 \\
    0 & -\tau M
\end{bmatrix}\begin{bmatrix}
    I \\ N(j\omega)
\end{bmatrix}\succeq 0
    \end{equation}
holds for all $\omega \in \mathbb{R}$, then $\textup{SRG}(H) \subseteq S(\Pi)$ with $\mathcal{S}(\Pi)$ as defined in \eqref{eq:SPI}.
\end{theorem}
\begin{proof}
    Pre- and post multiply the inequality in \eqref{eq:FDI} by $\Delta \hat{v}(j\omega)=\begin{bmatrix}
            \Delta \hat{u}(j\omega)^* &\Delta \hat{y}(j\omega)^*
        \end{bmatrix}^*$ and its conjugate transpose, with $\Delta \hat{u}(j\omega), \Delta \hat{y}(j\omega)$ the Fourier transforms of signals $\Delta u, \Delta y$ for which $u_i,y_i \in \mathcal{L}_2^n$, $i=\left\{1,2\right\}$, and integrate over $\omega$ from $\omega = -\infty$ to $\omega = +\infty$. Invoking the soft IQC in \eqref{eq:sIQC}, then implies that
        \begin{equation*}
            \int_{-\infty}^\infty \Delta \hat{v}(j\omega)^* (\Pi \otimes I)\Delta \hat{v}(j\omega)d\omega \geq 0.
        \end{equation*}
  Application of Parseval's relation shows that for all $(u_1,y_1),(u_2,y_2) \in G(H)$ the time-domain inequality
\begin{equation}\label{eq:dp}
        \langle \Delta v, \Pi \Delta v\rangle \geq 0,
    \end{equation}
    representing a static IQC, holds true. Direct application of Lemma~\ref{lem:SRGIQC} yields the result. 
\end{proof}

When $H$ is restricted to be LTI, the inequality in \eqref{eq:FDI} can be simplified by exploiting the relation $\Delta \hat{y}(j\omega) = H(j\omega) \Delta\hat{u}(j\omega)$, as formalized in the next corollary.

\begin{corollary}\label{col:sSRG1}
    Let $H$ be a bounded LTI system with transfer function matrix $H(s) \in \mathbb{C}^{n\times n}$. If there exists a symmetric matrix $\Pi \in \mathbb{R}^{2\times 2}$ such that the frequency-domain inequality 
    \begin{equation}
        \begin{bmatrix}
            I\\H(j\omega)
        \end{bmatrix}^*\left(\Pi \otimes I\right)\begin{bmatrix}
           I\\ H(j\omega)
        \end{bmatrix}\succeq 0
    \end{equation}holds for all $\omega \in \mathbb{R}$, then $\textup{SRG}(H) \subseteq S(\Pi)$ with $\mathcal{S}(\Pi)$ as defined in \eqref{eq:SPI}.
\end{corollary}
\begin{proof}
    Take $N(j\omega) = \textup{diag}( H(j\omega),I)$ and $M = \left[\begin{smallmatrix}
        1 & -1\\ -1 & 1
    \end{smallmatrix}\right]$, and note that the soft IQC in \eqref{eq:sIQC} holds with an equality in this case. As such, the condition $\tau \geq 0$ in Theorem~\ref{th:sSRG1} can be relaxed to $\tau \in \mathbb{R}$. Using $\Delta \hat{v}(j\omega) = [I, H(j\omega)^*]^*\Delta \hat{u}(j\omega)$ and applying Finsler's lemma \cite{Meijer} yield the result.
\end{proof}
Corollary~\ref{col:sSRG1} provides a frequency-domain equivalent to the statements in \cite[Theorem 9]{Groot} linking dissipativity LMIs and IQCs to soft SRGs for LTI systems. 

\subsection{Overbounding hard SRGs}
In a similar spirit as for soft SRG computations, we provide a key result for overbounding hard SRGs.
\begin{theorem}\label{th:hSRG1}
      Consider a causal system $H:\mathcal{L}_{2e}^n\to \mathcal{L}_{2e}^n$ and suppose that this system satisfies a hard incremental IQC of the form as in \eqref{eq:hIQC} with multipliers $(M,N)$, where $N(s)=C_N(sI-A_N)^{-1}B_N+D_N$ is causal and bounded. If there exist a symmetric matrix $\Pi \in\mathbb{R}^{2\times 2}$, a matrix $P=P^\top \succeq 0$, and a number $\tau> 0$ such that the LMI
  \begin{equation}\label{eq:LMI1}
        \begin{bmatrix}
            A_N^\top P+PA_N & PB_N \\
            B_N^\top P &-\Pi\otimes I
        \end{bmatrix}+\tau V^\top M V\preceq 0,
    \end{equation}
   with $V = [C_N,\; D_N]$ holds, then $\textup{SRG}_e(H) \subset S(\Pi)$ with $\mathcal{S}(\Pi)$ as defined in \eqref{eq:SPI}.
\end{theorem}
\begin{proof}
  Define signals $g = N \Delta v$, with $\Delta v = v_1-v_2$, $v_i =(u_i,y_i) \in G_e(H)$, $i=\left\{1,2\right\}$. Note that the LTI system $N$ admits the state-space representation
  \begin{equation}
      N:\begin{cases}
          \dot{x}_N&=A_Nx_N+B_N\Delta v,\quad x_N(0)=0,\\
          g&=C_Nx_N+D_N\Delta v.
      \end{cases}
  \end{equation}
Pre- and post-multiplication of \eqref{eq:LMI1} with $[x_N^\top, \Delta v^\top]^\top$ yields
\begin{equation*}
    \frac{d}{dt}W(x_N(t)) +\tau  g(t)^\top Mg(t)\leq \Delta v(t)^\top(\Pi \otimes I)\Delta v(t),
\end{equation*}
where $W(x_N) = x_N^\top Px_N$. Integrating the above inequality over $[0,T]$ and using both $W(x_N)\geq 0$ and $W(0)=0$ yields
\begin{equation*}
    \tau \int_0^Tg(t)^\top Mg(t)dt \leq \int_0^T \Delta v(t)^\top(\Pi \otimes I)\Delta v(t)dt.
\end{equation*}
As the left-hand side of the above inequality is non-negative for all $v_1,v_2 \in G_e(H)$ due to the hard IQC, it follows that $\int_0^T \Delta v(t)^\top(\Pi \otimes I)\Delta v(t)dt\geq 0$. Application of Lemma~\ref{lem:SRGIQC} finishes the proof. 
\end{proof}

We highlight several important aspects of both Theorem~\ref{th:sSRG1} and Theorem~\ref{th:hSRG1}.
\begin{enumerate}
    \item The incremental soft and hard IQCs are used via the S-procedure to represent the system $H$, i.e., $M$ and $N$ embed system characteristics in both the frequency-domain inequality \eqref{eq:FDI} and the LMI \eqref{eq:LMI1}, thereby avoiding the need for a complex system model that could hinder efficient computations.
    \item The matrix $\Pi$ is a free variable that defines the geometry of the SRG overbounding region according to \eqref{eq:SPI}. Hence, its role is to embed the SRG overbounding structure in the inequalities. 
    \item When $N$ in Theorem~\ref{th:sSRG1} is an LTI filter that admits the transfer function realization in \eqref{eq:Ns} with the pair $(A_N,B_N)$ controllable, the frequency-domain inequality \eqref{eq:FDI} can be recast via the Kalman-Yakubovich-Popov lemma \cite{Rantzer} as the LMI in \eqref{eq:LMI1}, with the difference that the extra condition $P\succeq 0$ can be dropped.
    \item In contrast to Lemma~\ref{lem:SRGIQC}, Theorems~\ref{th:sSRG1} and~\ref{th:hSRG1} do not establish SRG-fullness; they only show that the class of systems characterized by a dynamic IQC has its SRG contained in $\mathcal{S}(\Pi)$, but whether the converse is true remains open. On the other hand, class membership can be \emph{excluded}: if $\textup{SRG}(H)/\textup{SRG}_e(H)\not\subseteq\mathcal{S}(\Pi)$, then $H$ does not satisfy the respective soft/hard dynamic IQC. 
\end{enumerate}

\subsection{Overbound refinements}
To systematically refine overbounds of $\textup{SRG}(H)$ and $\textup{SRG}_e(H)$ based on Theorem~\ref{th:sSRG1} and Theorem~\ref{th:hSRG1}, respectively, we search for a number of matrices $\Pi$ that satisfy \eqref{eq:FDI} and \eqref{eq:LMI1}. To this end, we parametrize $\Pi$ as 
\begin{equation}\label{eq:PI}
    \Pi(\sigma, \lambda, r) := \sigma \begin{bmatrix}
       r^2- \lambda^2 & \lambda \\
        \lambda & -1
    \end{bmatrix},
\end{equation}
with $\sigma \in \left\{-1, +1\right\}$, $\lambda\in \mathbb{R}$, and $r >0$, see also \cite{Groot} for a further motivation. As $\det(\Pi)=-r^2<0$, for this choice of $\Pi$ the region $\mathcal{S}(\Pi)$ in \eqref{eq:SPI} corresponds to either the interior $(\sigma=+1)$ or the exterior $(\sigma = -1)$ of a disk in the complex plane, centered at $\lambda$ on the real axis with radius $r$. In addition, we let \begin{equation}
    \Pi(\sigma,\lambda, \infty):= \begin{bmatrix}
    0 & \lambda \\
    \lambda & 0
\end{bmatrix},
\end{equation} to strictly allow for half-planes. For each $\lambda \in \Lambda \subseteq \mathbb{R}$ we minimize $\sigma r^2$ subject to \eqref{eq:FDI} respectively \eqref{eq:LMI1} and collect the resulting matrices $\Pi$ in the sets $\mathbf{\Pi}(M,N)$ and $\mathbf{\Pi}_e(M,N)$. Intersecting the associated regions $\mathcal{S}(\Pi)$ yields the overbounds 
    \begin{align}
        \textup{SRG}_\circ(H) &\subseteq \bigcap_{\Pi \in \mathbf{\Pi}_\circ(M,N)}\mathcal{S}(\Pi),
    \end{align}
where $\circ$ refers to either the soft or hard case. If $H$ is known to satisfy multiple soft or hard IQCs, we can repeat the above procedure for each individual IQC and/or linear combinations thereof to tighten the overbounding regions.  

\begin{remark}
    \normalfont  In the special case where $H$ is LTI and satisfies $H(s)^*H(s) = H(s)H(s)^*$, it follows from Corollary~\ref{col:sSRG1} and \cite[Theorem 12]{Groot} that applying the refinement procedure with $\Lambda = \mathbb{R}$ yields the exact SRG. 
\end{remark}

\begin{remark}
\normalfont \normalfont
In parallel with the analysis of incremental properties via SRGs and incremental IQCs, the same constructions can be specialized to \emph{non-incremental} properties by fixing one input at the origin, e.g., $u_2=0$. This specialization leads to the notion of scaled graphs (SGs) and non-incremental IQCs, see also \cite[Section III.B]{Chaffey23}.
\end{remark}

\subsection{Illustrative example}\label{sec:ex}
Consider a system that simultaneously satisfies a static incremental soft IQC with multipliers
\begin{equation}\label{eq:SMN}
    M = \begin{bmatrix}
        0.4 & 0.5 \\
        0.5 & -2
    \end{bmatrix}, \quad \textup{and} \quad N=I,
\end{equation}
and a dynamic incremental soft IQC with multipliers
\begin{equation}\label{eq:HMN}
  M= \begin{bmatrix}
       0 & 1 \\ 1 & 0
   \end{bmatrix}, \quad \textup{and} \quad N(s) = \begin{bmatrix}
    1 & - 0.5\\
    0 & \tfrac{s+1}{s+2} 
\end{bmatrix}.
\end{equation}
Note that \eqref{eq:HMN} describes a specifc Zames-Falb multiplier \cite{Carrasco}. Explicit examples of systems satisfying these IQCs include:
\begin{enumerate}[label=\alph*)]
    \item An LTI system with transfer function 
     \begin{equation}\label{eq:LTI_example}
    H(s) = 0.13\tfrac{s+2}{(s+0.6)^2};\end{equation}
    \item A static nonlinearity that satisfies the slope restriction
    \begin{equation}
          0\leq \frac{Hu_1-Hu_2}{u_1-u_2} \leq \frac{1}{2} \quad \textup{ for all } u_1,u_2\in \mathbb{R}.
    \end{equation}
\end{enumerate}
To generate an overbound on the soft SRG of systems satisfying these IQCs, we apply the above procedure in which we solve the LMI in \eqref{eq:LMI1} (excluding the condition $P\succeq 0$). The result is shown in Fig.~\ref{fig:SRGexample} in grey. The region obtained with the static IQC in \eqref{eq:SMN} alone is represented by the hatched region. It can be seen that adding the dynamic IQC significantly improves the overbounding region. The SRG of the LTI system \eqref{eq:LTI_example} is shown in blue for reference. 

\begin{figure}[h]
	\centering	\includegraphics[width=0.7\linewidth]{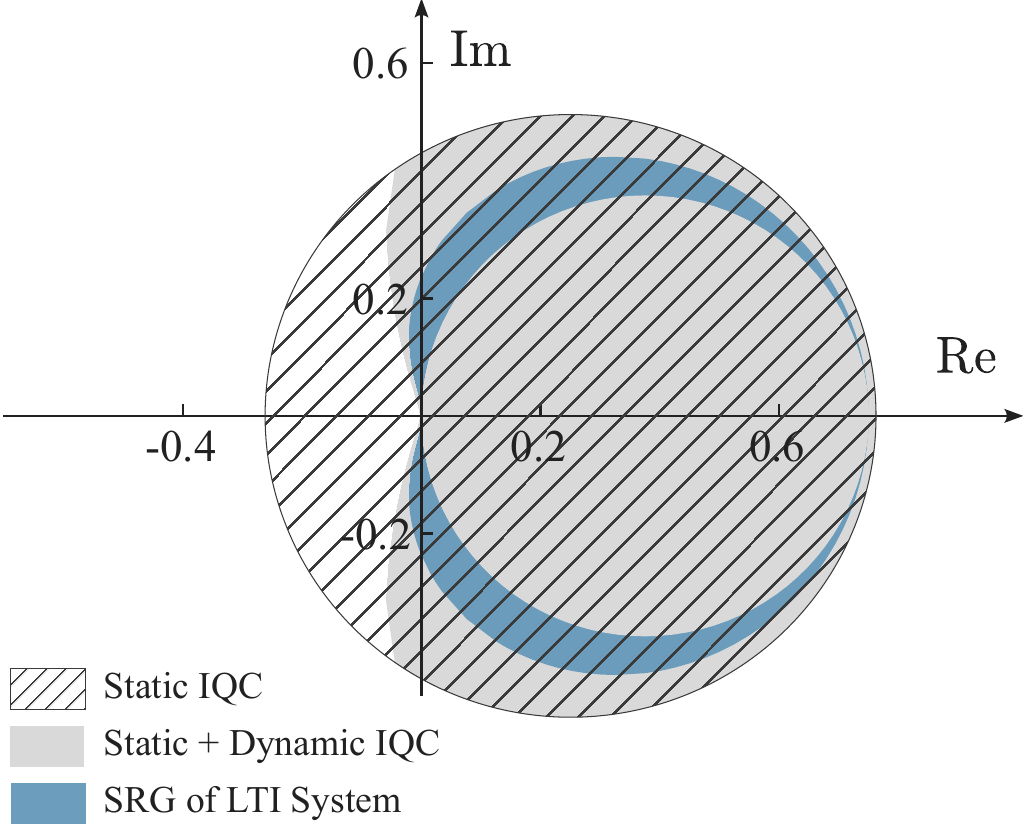}
	\caption{Overbounds on the soft SRG for systems satisfying the IQCs in \eqref{eq:SMN} and \eqref{eq:HMN} computed using: a) only the static IQC (hatched region), and b) both the static and dynamic IQCs. The SRG of the LTI system in \eqref{eq:LTI_example} is shown in blue.}
	\label{fig:SRGexample}
\end{figure}

\section{Feedback stability analysis}\label{sec:AS}
In this section, we exploit the above results for analysis of the feedback interconnection of an LTI system and a (dynamic) nonlinearity, also referred to as a Lur'e system.

\subsection{Lur'e system description}
Consider the feedback interconnection shown in Fig.~\ref{fig:FB}, where $H_1:\mathcal{L}_{2e}^n \to \mathcal{L}_{2e}^n$ is a causal LTI system, and $H_2 :\mathcal{L}_{2e}^n \to \mathcal{L}_{2e}^n$ is a causal nonlinear system. The functional equations describing the interconnection are given by
\begin{align}\label{eq:system}
\begin{cases}
u_1 = w + y_2,\\
u_2 = y_1,
\end{cases}
\quad \text{and} \quad\:
\begin{cases}
y_1 = H_1 u_1,\\
y_2 = H_2 u_2,
\end{cases}
\end{align} 
where $w \in \mathcal{L}_{2}^n$ are external signals, and  $u_1,y_1,u_2,y_2$ are internal signals. In the remainder, we assume that the feedback system is well-posed, that is, the mapping $u_1\mapsto w$ has a causal inverse on $\mathcal{L}_{2e}^n$. 

\tikzset{
    block/.style = {draw, rectangle,
        minimum height=1cm,
        minimum width=1cm},
    input/.style = {coordinate,node distance=1cm},
    output/.style = {coordinate,node distance=1cm},
    arrow/.style={draw, -latex,node distance=2cm},
    sum/.style = {draw, circle, node distance=.5cm, inner sep=.3em},
}      

\begin{figure}[htbt!]
\centering
\setlength{\unitlength}{1mm}
\begin{picture}(50,25)
\thicklines \put(0,20){\vector(1,0){8}} \put(10,20){\circle{4}}
\put(12,20){\vector(1,0){8}} \put(20,15){\framebox(10,10){$H_1$}}
\put(30,20){\vector(1,0){16}} \put(40,20){\line(0,-1){15}}
\put(40,5){\vector(-1,0){10}}  
 \put(20,0){\framebox(10,10){$H_2$}}
\put(20,5){\line(-1,0){10}} \put(10,5){\vector(0,1){13}}
\put(13,0){\makebox(5,5){$y_2$}} \put(32,20){\makebox(5,5){$y_1$}}
\put(0,20){\makebox(5,5){$w$}}  
\put(13,20){\makebox(5,5){$u_1$}} \put(32,0){\makebox(5,5){$u_2$}}
\end{picture} 
  \caption{Lur'e feedback interconnection.}
        \label{fig:FB}
\end{figure}
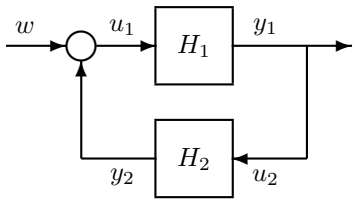 

We are interested in graphically showing that the well-posed interconnection in Fig.~\ref{fig:FB} is {bounded}  in the sense that the mapping $T:=w\mapsto y_1$ satisfies $\|T\|_\Delta <\infty$.

\subsection{An SRG--IQC stability result}
The result below establishes boundedness of the Lur'e system in Fig.~\ref{fig:FB} based on the separation of a frequency-wise SRG and an SRG overbound.
\begin{theorem}\label{th:fb}
     Consider the LTI system $H_1:\mathcal{L}_{2}^n \to \mathcal{L}_{2}^n$ and the nonlinearity $H_2:\mathcal{L}_2^n \to \mathcal{L}_2^n$, and suppose that the interconnection in Fig.~\ref{fig:FB} between $H_1$ and $\mu H_2$ is well-posed for all $\mu \in (0,1]$. In addition, suppose that $\mu H_2$ satisfies the soft IQC in \eqref{eq:sIQC} with multipliers $(M,N)$ for all $\mu \in (0,1]$. If there exists $r>0$ such that for all $\mu \in (0,1]$
    \begin{equation}\label{eq:dist}
        \textup{dist}(\textup{SRG}(H_1(j\omega)), \mu \mathcal{C}(M,N)) \geq r,
    \end{equation}
    with
    \begin{equation}
        \mathcal{C}(M,N) = \bigcap_{\Pi \in \mathbf{\Pi}(M,N)}\left\{z \in \mathbb{C}\left| \begin{bmatrix}
            z \\ 1
        \end{bmatrix}^*\Pi \begin{bmatrix}
            z \\ 1
        \end{bmatrix}\geq 0\right.\right\},
    \end{equation}
    and where $\mathbf{\Pi}(M,N)$ is the set of all matrices $\Pi$ satisfying \eqref{eq:FDI}, then $\|T\|_\Delta < \infty$. 
\end{theorem}

\begin{proof}
    Under the separation condition \eqref{eq:dist} stated in the theorem, it follows that for each $\omega \in \mathbb{R}$, there must be at least one matrix $\Pi \in \mathbf{\Pi}(M,N)$ such that
    \begin{equation}
        \begin{bmatrix}
            z(j\omega)\\1
        \end{bmatrix}^*\Pi \begin{bmatrix}
            z(j\omega) \\ 1
        \end{bmatrix}\leq  -\epsilon I,
    \end{equation}
    with $z(j\omega) \in \textup{SRG}(H_1(j\omega))$ and $\epsilon >0$ a uniform constant that depends on $r$. Exploiting the fact that $z(j\omega)=\rho(u,H(j\omega)u)e^{\pm j\theta(u,H(j\omega)u)}$, it then follows from a similar result as in Lemma~\ref{col:sSRG1} that 
    \begin{equation}
        \begin{bmatrix}
            H_1(j\omega)\\I
        \end{bmatrix}^*\left(\Pi \otimes I\right)\begin{bmatrix}
            H_1(j\omega) \\ I
        \end{bmatrix}\preceq  -\epsilon I.
    \end{equation}
    Since $\Pi \in \mathbf{\Pi}(M,N)$ it follows that
    \begin{equation}
        x^*(\Pi \otimes I - \tau N(j\omega)^*MN(j\omega))x \geq 0
    \end{equation}
    for all $x \in \mathbb{C}^{2n}$, and some $\tau > 0$. Taking $x=[H_1(-j\omega), I]^*u$ with $u \in \mathbb{C}^n$, we find
    \begin{equation}\label{eq:IQC}
        \begin{bmatrix}
            H_1(j\omega) \\ I
        \end{bmatrix}^* N(j\omega)^*MN(j\omega)\begin{bmatrix}
            H_1(j\omega) \\I
        \end{bmatrix}\leq -\tfrac{\epsilon}{\tau} I
    \end{equation}
    for all $\omega \in \mathbb{R}$, so that all conditions of the IQC theorem \cite[Theorem 1]{Megretski} are satisfied, and the system is stable. 
\end{proof}

Theorem~\ref{th:fb} provides a connection between SRGs and the IQC theorem (Theorem 1 in \cite{Megretski}). In particular, Condition~\eqref{eq:dist} in Theorem~\ref{th:fb} is an SRG separation condition that admits a graphical verification, similar as the conditions in \cite[Theorem 2]{Chen}. At the same time, the proof of Theorem~\ref{th:fb} relies fundamentally on the IQC theorem (see \eqref{eq:IQC}). In fact, Theorem~\ref{th:fb} provides a new graphical interpretation of IQC-based stability conditions, in which stability can be assessed \emph{without explicitly invoking multipliers in the graphical test itself.} For example, in the SISO setting, condition \eqref{eq:dist} reduces to a ``forbidden'' region in the complex plane that the Nyquist curve of $H_1$ must avoid. This mirrors the classical geometric interpretations arising from static IQC tests, such as the small-gain theorem and the circle-criterion, whereas dynamic IQCs typically lead to conditions expressed in terms of the modified Nyquist contour of $NH_1$. 

The connection between SRG and IQC stability conditions becomes exact when restricting to a single static multiplier. In this case, Theorem~\ref{th:fb} is equivalent to the IQC theorem of \cite{Megretski}. For instance, $\Pi=\mathrm{diag}(\gamma^2,-1)$ with $\gamma >0$ recovers the small-gain theorem, whereas $\Pi=\left[\begin{smallmatrix}
    0 &k\\
    k & -2
\end{smallmatrix}\right]$ with $k>0$ yields the circle-criterion. When dynamic multipliers are considered, Theorem~\ref{th:fb} may not always go beyond the static case. This will be illustrated next through Popov multipliers.

\subsection{Popov multipliers}
The Popov criterion \cite{Jonsson} applies to memoryless, sector-bounded nonlinearities that satisfy a static IQC with 
\begin{equation}\label{eq:Ps}
    M = \begin{bmatrix}
        0 & k \\
        k & -2
    \end{bmatrix}, \:\:k>0, \quad \textup{and} \quad N=I,
\end{equation}
and a soft dynamic IQC (satisfied with equality) with 
\begin{equation}\label{eq:Pd}
    M = \begin{bmatrix}
        0 & 1\\
        1 & 0 
    \end{bmatrix}, \quad \textup{and} \quad N(s) = \begin{bmatrix}
        \tfrac{s}{1+\sigma s} & 0 \\
        0 & 1
    \end{bmatrix},\: \sigma \geq 0.
\end{equation}
The linear combination of the multipliers in \eqref{eq:Ps} and \eqref{eq:Pd} is also known as a Popov multiplier. To compute an overbound on S(R)Gs associated with Popov multipliers, we invoke Theorem~\ref{th:sSRG1}. Focusing on the SISO case ($n=1)$, the frequency-domain inequality in \eqref{eq:FDI} reduces to
\begin{equation}\label{eq:Pfdi}
    \Pi - \tau_1 \begin{bmatrix}
        0 & k \\
        k & -2
    \end{bmatrix}- \tau_2\begin{bmatrix}
        0 & \tfrac{j\omega}{1+\sigma j\omega}\\
        \tfrac{-j\omega}{1-\sigma j\omega} & 0
    \end{bmatrix}\succeq 0,
\end{equation}
with $\tau_1\geq 0$ and $\tau_2 \in \mathbb{R}$. Since \eqref{eq:Pfdi} must hold for all $\omega \in \mathbb{R}$, in particular $\omega = 0$, it follows that $\Pi$ must admit the form
\begin{equation}
    \Pi = \tau_1 \begin{bmatrix}
        0 & k \\
        k & -2
    \end{bmatrix} + Q,\quad \textup{with}\quad Q\succeq 0.
\end{equation}
The smallest admissible overbounding region $\mathcal{S}(\Pi)$ is therefore obtained by setting $Q=0$, which coincides exactly with the region induced by the static IQC alone. Consequently, including the dynamic IQC does not tighten the resulting S(R)G overbound beyond the static case. 

This outcome may appear a limitation of the proposed framework. However, it reflects a structural limitation rather than a technical one. Indeed, it follows from \cite[Proposition 12]{Chaffey23} that the exact SRG of a broad class of static, sector-bounded nonlinearities is precisely the disc resulting from the static multiplier, and, therefore, dynamic multipliers will not reduce this region. The key implication is that the Popov criterion cannot be recovered from Theorem~\ref{th:fb} when working directly with the SRGs of the original feedback components. Of course, the Popov result is recovered after a loop transformation that shifts the multiplier from $H_2$ to $H_1$, similar to what is done in standard passivity results.

As a final note, we emphasize that incorporating dynamic IQCs  for SRG overbounds can still yield improvements over the static case, which was illustrated by the example in Section~\ref{sec:ex}. The improvements essentially arise from further restricting the class of nonlinearities, for example to monotone operators, for which Zames–Falb multipliers can yield tighter SRG characterizations.

\section{Conclusions}\label{sec:conclusions}
This paper presents new tools for computing overbounds on soft and hard SRGs of nonlinear systems that satisfy dynamic IQCs. The IQCs serve as simpler characterizations of the system, replacing exact descriptions to enable tractable computational methods. We use the IQCs in an S-procedure manner to develop frequency-domain and LMI-based methods for constructing SRG overbounds. These overbounds are exploited in a feedback stability theorem for general Lur'e systems, which offers a new graphical interpretation of classical IQC-based stability conditions that does not directly involve multipliers in the graphical test itself. 
\bibliographystyle{IEEEtran}

\end{document}